\documentclass[a4paper,openbib,12pt]{article}
\usepackage{geometry}
\usepackage{amsmath}
\usepackage{amssymb}
\usepackage{amsfonts}

\newtheorem{theorem}{Theorem}[section]

\newtheorem{conjecture}[theorem]{Conjecture}
\newtheorem{corollary}[theorem]{Corollary}

\newtheorem{example}[theorem]{Example}

\newtheorem{lemma}[theorem]{Lemma}

\newtheorem{problem}[theorem]{Problem}

\newtheorem{remark}[theorem]{Remark}

\newenvironment{proof}[1][Proof]{\noindent\textbf{#1.} }{\ \rule{0.5em}{0.5em}}
\input{tcilatex}
\begin{document}

\title{Symmetries of flat manifolds, Jordan property and the general Zimmer
program}
\author{Shengkui Ye}
\maketitle

\begin{abstract}
We obtain a sufficient and necessary condition for a finite group to act
effectively on a closed flat manifold. Let \ $G=E_{n}(R)$, $EU_{n}(R,\Lambda
),$ $\mathrm{SAut}(F_{n})$ or $\mathrm{SOut}(F_{n}).$ As applications, we
prove that when $n\geq 3$ every group action of $G$ on a closed flat
manifold $M^{k}$ ($k<n$) by homeomorphisms is trivial. This confirms a
conjecture related to Zimmer's program for flat manifolds. Moreover, it is
also proved that the group of homeomorphisms of closed flat manifolds are
Jordan with Jordan constants depending only on dimensions.
\end{abstract}


\section{Introduction}

Let $\Gamma $ be a group acting freely, isometrically, properly
discontinuously and cocompactly on the Euclidean space $\mathbb{R}^{n}.$ The
quotient space $M=\mathbb{R}^{n}/\Gamma $ is called a flat manifold. A
classical result of Bieberbach implies that there is a short exact sequence
of groups%
\begin{equation*}
1\rightarrow \mathbb{Z}^{n}\rightarrow \Gamma \rightarrow \Phi \rightarrow 1,
\end{equation*}%
where $\Phi <\mathrm{GL}_{n}(\mathbb{Z})$ is called the holonomy group of $%
M. $ In this article, we are interested in the following question on
symmetries of $M$ :

\begin{problem}
What kind of finite groups $G$ could act effectively on a flat manifold $M$
by homeomorphisms?
\end{problem}

We give a sufficient and necessary condition for a finite group to act
effectively on a flat manifold and thus solve the previous problem
completely.

\begin{theorem}
\label{th3} A finite group $G$ acts effectively on a closed flat manifold $%
M^{n}$ with the fundamental group $\pi $ and the holonomy group $\Phi $ by
homeomorphisms if and only if there is an abelian extension 
\begin{equation*}
1\rightarrow A\rightarrow G\rightarrow Q\rightarrow 1
\end{equation*}%
such that

\begin{enumerate}
\item[(i)] $Q\cong \Phi ^{\ast }/\Phi $ for a finite subgroup $\Phi ^{\ast }<%
\mathrm{GL}_{n}(\mathbb{Z});$

\item[(ii)] there is an $(\Phi ^{\ast },Q)$-equivariant surjection $\ \alpha
:\mathbb{Z}^{n}\twoheadrightarrow A,$ and a commutative diagram%
\begin{equation*}
\begin{array}{ccccccccc}
1 & \rightarrow & \mathbb{Z}^{n} & \rightarrow & G^{\ast } & \rightarrow & 
\Phi ^{\ast } & \rightarrow & 1 \\ 
&  & \alpha \downarrow &  & f\downarrow &  & \downarrow &  &  \\ 
1 & \rightarrow & A & \overset{i}{\rightarrow } & G & \rightarrow & Q & 
\rightarrow & 1%
\end{array}%
\end{equation*}%
with torsion-free kernel $\ker f=\pi .$ Here \ $\alpha (gx)=\bar{g}\alpha
(x) $ for any $x\in \mathbb{Z}^{n},g\in \Phi ^{\ast }$, where $\bar{g}\in Q$
acts on the abelian group $A$ through the exact sequence and $\ker (\Phi
^{\ast }\rightarrow Q)=\Phi $.
\end{enumerate}
\end{theorem}

Denote by $\mathrm{Aff}(M)$ the group of affine equivalences of the closed
flat manifold $M$ and by $\mathrm{Aff}_{0}(M)$ the identity component, which
is a torus of dimension $b_{1}(M)$ (the first Betti number). Charlap and
Vasquez \cite{cv} prove that $\mathrm{Aff}(M)/\mathrm{Aff}_{0}(M)$ is
isomorphic to the outer automorphism group $\mathrm{Out}(\pi _{1}(M)).$ From
this, it is not hard to derive necessary conditions on finite groups acting
on $M.$ For a fixed group homomorphism $\varphi :G\rightarrow \mathrm{Out}%
(\pi _{1}(M)),$ Lee and Raymond \cite{lr} obtain a sufficient and necessary
condition for the finite group $G$ acts effectively on $M$ inducing $\varphi 
$. However, the group $\mathrm{Out}(\pi _{1}(M))$ is generally complicated,
partly because the holonomy group $\Phi $ is subtle. For some open problems
relating $\mathrm{Out}(\pi _{1}(M)),$ see Szczepa\'{n}ski \cite{sz}. Our
characterization does not use $\mathrm{Out}(\pi _{1}(M)).$ The proof of
Theorem \ref{th3} depends on results obtained by Lee and Raymond \cite{lr}.

As an easy application of Theorem \ref{th3} to finite group actions on tori,
we have a simpler characterization as the following. To the best of our
knowledge, this characterization has so far not been stated explicitly in
the literature, except possibly for low-dimensional cases (eg. $n=1,2$).

\begin{theorem}
A finite group $G$ acts effectively on a torus $T^{n}$ if and only if there
is an abelian extension 
\begin{equation*}
1\rightarrow A\rightarrow G\rightarrow Q\rightarrow 1
\end{equation*}%
such that

\begin{enumerate}
\item[(i)] $Q<\mathrm{GL}_{n}(\mathbb{Z});$

\item[(ii)] there is an $Q$-equivariant surjection $\ \alpha :\mathbb{Z}%
^{n}\twoheadrightarrow A$ and the cohomology class representing of the
extension lies in the image $\func{Im}(H^{2}(Q;\mathbb{Z}^{n})\rightarrow
H^{2}(Q;A)).$
\end{enumerate}
\end{theorem}

We give two applications in the following. Let $\mathrm{SL}_{k}(\mathbb{Z})$
be the special linear group over integers. Since $\mathrm{SL}_{k}(\mathbb{Z}%
) $ acts linearly on the Euclidean space $\mathbb{R}^{k}$ fixing the origin,
there is an induced action of $\mathrm{SL}_{k}(\mathbb{Z})$ on the sphere $%
S^{k-1}.$ It is believed that this action is minimal in the following sense.

\begin{conjecture}
\label{conj}Any group action of $\mathrm{SL}_{n}(\mathbb{Z})$ $(n\geq 3)$ on
a compact connected $r$-manifold by homeomorphisms factors through a finite
group action if $r<n-1.$
\end{conjecture}

The smooth version of this conjecture was formulated by Farb and Shalen \cite%
{fs}, that is an analogue of a special case of one of the central
conjectures in the Zimmer program (see \cite{zi3,zi5}) concerning group
actions of lattices in Lie groups on manifolds. For more details and the
status, see the survey articles \cite{fi,zm}.

As an application of Theorem \ref{th1}, we confirm Conjecture \ref{conj} for
closed flat manifolds, as a special case of the following. Recall from
Section \ref{sec} the definitions of elementary linear group $E_{n}(R)$ over
an associative ring $R,$ the elementary unitary group $EU_{n}(R,\Lambda )$
over a form ring $(R,\Lambda ),$ the special automorphism group $\mathrm{SAut%
}(F_{n})$ and the special outer automorphism group $\mathrm{SOut}(F_{n})$ of
a free group $F_{n}.$ Note that when $R=\mathbb{Z},$ we have \ $E_{n}(R)=%
\mathrm{SL}_{n}(\mathbb{Z}).$

\begin{theorem}
\label{th2}Let \ $G=E_{n}(R)$, $EU_{n}(R,\Lambda ),$ $\mathrm{SAut}(F_{n})$
or $\mathrm{SOut}(F_{n}),n\geq 3.$ Suppose that $M^{r}$ is a closed flat
manifold. When $r<n,$ any group action of the group \ $G$ on $M^{r}$ by
homeomorphisms is trivial, i.e. is the identity homeomorphism$.$
\end{theorem}

The ideas used in the proof of Theorem \ref{th2} differ from most other work
on Zimmer's program, where actions preserving volume as well as extra
geometric structures such as a connection or pseudo-Riemannian metric are
studied (see \cite{f1,f2,zi2,zi3,zi4,zi5}). The latter work uses
ergodic-theoretic methods, while our methods are topological. Based on the
knowledge on finite groups acting on flat manifolds, we prove that the
alternating group $A_{n+1}$ $(n\geq 4)$ cannot act effectively on $M^{r}$
when $r<n$ (see Theorem \ref{th1}). The group $G$ is normally generated by $%
A_{n+1}$ and thus acts trivially on $M.$ When $M^{r}=S^{1},$ the circle, the
case of $\mathrm{SL}_{n}(\mathbb{Z})$ in Theorem \ref{th2} (more generally
lattices in semisimple Lie groups) is already known to \cite{wi,bm,gy}.
Weinberger \cite{we} obtains a similar result for the torus $M=T^{r}.$
Bridson and Vogtmann \cite{bv1} prove a similar result for $\mathrm{SAut}%
(F_{n})$ and $\mathrm{SOut}(F_{n})$ actions on spheres. When $r\leq 5,$ the
case of $\mathrm{SL}_{n}(\mathbb{Z})$ in Theorem \ref{th2} is proved by Ye 
\cite{ye}. For orientable manifolds whose Euler characteristics are not
divisible by $6,$ Conjecture \ref{conj} is confirmed by Ye \cite{ye2}. Note
that the Euler characteristic of a flat manifold is always zero and the
result proved in \cite{ye2} cannot apply. For $C^{1+\beta }$-group actions
of finite-index subgroup in $\mathrm{SL}_{n}(\mathbb{Z}),$ one of the
results proved by Brown, Rodriguez-Hertz and Wang \cite{brw} confirms
Conjecture \ref{conj} for surfaces. For $C^{2}$-group actions of cocompact
lattices, Brown-Fisher-Hurtado \cite{bfh} confirms Conjecture \ref{conj}.
Note that the $C^{0}$-actions could be very different from smooth actions.
For some unsmoothable group actions of mapping class groups and $\mathrm{Out}%
(F_{n})$ on one-dimensional manifolds$,$ see Baik-Kim-Koberda \cite{bkt}.

\begin{remark}
The usual Zimmer's program\bigskip\ is stated for any lattice in high-rank
semisimple Lie groups. However, Theorem \ref{th2} do not hold for general
lattices. For example, the congruence subgroup $\Gamma (n,p),$ which is
defined as the kernel of $\mathrm{SL}_{n}(\mathbb{Z})\rightarrow \mathrm{SL}%
_{n}(\mathbb{Z}/p)$ for a prime $p,$ has a nontrivial finite cyclic quotient
group (cf. \cite{ls}, Theorem 1.1). The group $\Gamma (n,p)$ could act on $%
S^{1}$ through the cyclic group by rotations.
\end{remark}

We give another application as the following. Recall that a group $H$ is
called Jordan if there is a constant $c$ depending only on $H$ such that
every finite subgroup $G<H$ contains an abelian subgroup $A<G$ of index $%
|G,A|<c.$ Since the general linear group $\mathrm{GL}_{n}(\mathbb{Z})$
contains only finitely many conjugacy classes of finite groups, we see that $%
\mathrm{GL}_{n}(\mathbb{Z})$ is Jordan. A famous theorem of Jordan \cite{J}
(see also \cite{cr}) is that the general linear group $\mathrm{GL}_{n}(%
\mathbb{C})$ is Jordan. Ghys \cite{fi, gy2} conjectures that the
diffeomorphism group of any smooth compact manifold is Jordan. Zimmermann 
\cite{zim} proves the conjecture for compact 3-manifolds and Mundet i Riera 
\cite{ri1,ri2,ri3} proves the conjecture for tori, acyclic manifolds,
homology spheres, and manifolds with non-zero Euler characteristic. However,
it is shown by Csik\'{o}s, Pyber and Szab\'{o} \cite{cz} that the
diffeomorphism group of $S^{2}\times T^{2}$ is not Jordan. Denote by the
constant $f(n)=\max \{|G|:G<\mathrm{GL}_{n}(\mathbb{Z})$ is finite\}. The
study of $f(n)$ has a long history. Some arguments of Weisfeiler and Feit
would imply that $f(n)\leq 2^{n}(n+1)!$ when $n\geq 11.$ (for more details,
see the survey article \cite{gl}, Section 6.1). As another application of
Theorem \ref{th3}, we prove that homeomorphism groups of flat manifolds are
Jordan, as follows.

\begin{corollary}
\label{th4}Let $M^{n}$ be a closed flat manifold and $G$ be a finite
subgroup of the group of homeomorphisms of $M.$ Then $G$ contains an abelian 
\emph{normal} subgroup of index at most $f(n).$
\end{corollary}

The smooth version of Corollary \ref{th4} is already known to Mundet i Riera 
\cite{ri1}. Compared with his result, Corollary \ref{th4} holds true for
homeomorphisms and contains an explicit bound of subgroup index. Such a
bound would be crucial in the proof of Theorem \ref{th1}.

The paper is organized as follows. In Section 2, we give several basic facts
on minimal dimensions. In Section 3, the symmetries of flat manifolds are
studied and Theorem \ref{th3} is proved. Theorem \ref{th1} is proved in
Section 4, with the help of computer programs. In the last two sections, we
discuss the (infinite) groups normally generated by alternating groups and
prove Theorem \ref{th2}.

\section{Minimal dimensions}

Let $R=\mathbb{Z},\mathbb{Q}$ or $\mathbb{R}$ be the ring of integers,
rational numbers or real numbers, and $\mathrm{GL}_{n}(R)$ the general
linear group over $R$. For a group $G,$ define the minimal faithful
representation dimension 
\begin{equation*}
d_{R}(G)=\min \{n\mid G\hookrightarrow \mathrm{GL}_{n}(R)\}.
\end{equation*}%
Let $\mathcal{M}=\cup _{i=1}^{+\infty }\mathcal{M}^{i}$ denote a collection
of manifolds, and $\mathcal{M}^{i}$ denote the subset of $\mathcal{M}$
consisting of those manifolds of dimension $i$. For example, $\mathcal{R}%
=\cup _{i=1}^{+\infty }\mathbb{R}^{i}$ the set of Euclidean spaces, $%
\mathcal{T}=\cup _{i=1}^{+\infty }T^{i}$ the set of tori and $\mathcal{FM=}%
\cup _{i=1}^{+\infty }\mathcal{FM}^{i}$ the set of closed flat manifolds
(i.e. manifolds finitely covered by tori). For a group $G,$ define the
minimal acting dimensions%
\begin{equation*}
d_{h}(G,\mathcal{M})=\min \{n\mid G\hookrightarrow \mathrm{Homeo}(N)\text{
for some }N\in \mathcal{M}^{n}\}
\end{equation*}%
and%
\begin{equation*}
d_{s}(G,\mathcal{M})=\min \{n\mid G\hookrightarrow \mathrm{Diff}(N)\text{
for some }N\in \mathcal{M}^{n}\},
\end{equation*}%
where $\mathrm{Homeo}(N)$ (resp. $\mathrm{Diff}(N)$) is the group of
homeomorphisms (resp. diffeomorphisms) of $N.$ If there are no such minimal $%
n$, we define the dimensions as $\infty .$ The study of the representation
dimensions has a long history (eg. \cite{cr}). It is obvious that 
\begin{equation*}
d_{\mathbb{R}}(G)\leq d_{\mathbb{Q}}(G)\leq d_{\mathbb{Z}}(G)
\end{equation*}
and $d_{h}\leq d_{s}$ for any group $G.$ It is interesting to compare the
representation dimensions with the acting dimensions. In this section, we
prove several basic facts on these dimensions.

\begin{lemma}
Let $G$ be a group with a subgroup $H<G.$ Then $d(H)\leq d(G)$ for $d\in
\{d_{\mathbb{R}},d_{\mathbb{Q}},d_{\mathbb{Z}},d_{s},d_{h}\}.$
\end{lemma}

\begin{proof}
This is obvious, since any injective map of \ $G$ restricts an injection on
the subgroup $H.$
\end{proof}

\begin{lemma}
Let $G$ be a finite group. Then $d_{\mathbb{Q}}(G)=d_{\mathbb{Z}}(G)<+\infty
.$
\end{lemma}

\begin{proof}
The group $G$ acts effectively on the vector space $\mathbb{Q[}G]$ by
permuting basis. This implies $d_{\mathbb{Q}}(G)<+\infty .$ It is well-known
that a finite subgroup in $\mathrm{GL}_{n}(\mathbb{Q})$ to conjugate to a
subgroup in $\mathrm{GL}_{n}(\mathbb{Z})$ (cf. \cite{se}, 1.3.1), which
gives $d_{\mathbb{Q}}(G)=d_{\mathbb{Z}}(G).$
\end{proof}

\begin{lemma}
\label{ke}Let $G$ be a group normally generated by a simple subgroup $H,$
i.e. every element $g\in G$ is a product of conjugates of elements in $H.$
When $n<d_{h}(H,\mathcal{M}),$ any group action of $G$ on $N^{n}$ by
homeomorphisms is trivial for any $N\in \mathcal{M}^{n}.$
\end{lemma}

\begin{proof}
Let $f:G\rightarrow \mathrm{Homeo}(N)$ be a group homomorphism. Since the
restriction $f|_{H}$ is not injective, the subgroup $(\ker f)\cap H$ is not
trivial, which is the whole group $H$ considering the fact that $H$ is
simple. Since $G$ is normally generated by $H,$ we have $\ker f=G.$
\end{proof}

\begin{lemma}
\label{prodlem}Suppose that $\mathcal{M}$ is closed under taking products,
i.e. any $N^{n}\in \mathcal{M}^{n},N^{m}\in \mathcal{M}^{n}$ implies the
product $N^{n}\times N^{m}\in \mathcal{M}^{n+m}.$ For any two groups $%
G_{1},G_{2},$ we have that 
\begin{equation*}
d_{h}(G_{1}\times G_{2},\mathcal{M})\leq d_{h}(G_{1},\mathcal{M}%
)+d_{h}(G_{2},\mathcal{M})
\end{equation*}
and 
\begin{equation*}
d_{s}(G_{1}\times G_{2},\mathcal{M})\leq d_{s}(G_{1},\mathcal{M}%
)+d_{s}(G_{2},\mathcal{M}).
\end{equation*}
\end{lemma}

\begin{proof}
Suppose that there are injections $f_{1}:G_{1}\hookrightarrow \mathrm{Homeo}%
(N_{1})$ and $f_{1}:G_{2}\hookrightarrow \mathrm{Homeo}(N_{2}).$ There is an
injection $G_{1}\times G_{2}\hookrightarrow \mathrm{Homeo}(N_{1})\times 
\mathrm{Homeo}(N_{2})\hookrightarrow \mathrm{Homeo}(N_{1}\times N_{2}).$
\end{proof}

\begin{example}
The inequality in Lemma \ref{prodlem} may be strict. For example, any cyclic
group $\mathbb{Z}/n$ acts freely on the circle $S^{1}$ by rotations. This
implies that $d_{h}(\mathbb{Z}/n,\mathcal{FM})=1.$ Since $d_{\mathbb{R}}(%
\mathbb{Z}/n)\geq 2$ when $n>2,$ the acting dimension on flat manifolds may
be strictly less than the minimal faithful real representation dimension.
\end{example}

The following gives a relation between the integral representation dimension
and the acting dimension on flat manifolds.

\begin{lemma}
\label{le1}Suppose that $\mathcal{M}$ contains the set of tori. For any
group $G,$ we have that $d_{s}(G,\mathcal{M})\leq d_{\mathbb{Z}}(G).$
\end{lemma}

\begin{proof}
Since the general linear group $\mathrm{GL}_{n}(\mathbb{Z})$ acts on the
Euclidean space $\mathbb{R}^{n}$ preserving the integral lattice $\mathbb{Z}%
^{n},$ there is an induced faithful action on the torus $T^{n}=\mathbb{R}%
^{n}/\mathbb{Z}^{n}.$ This implies that $d_{s}(G,\mathcal{M})\leq d_{\mathbb{%
Z}}(G).$
\end{proof}

\section{Symmetries of flat manifolds}

Let $M^{n}$ be a closed flat manifold, i.e. a closed manifold finitely
covered by the torus $T^{n}.$ Suppose that a finite group $G$ acts
effectively on $M$ by homeomorphisms. Lee and Raymond \cite{lr} prove that $G
$ can act on $M^{n}$ by affine diffeomorphisms$.$ This implies that 
\begin{equation*}
d_{h}(G,\mathcal{FM})=d_{s}(G,\mathcal{FM}).
\end{equation*}%
Recall from \cite{lr} the definitions of abstract crystallographic and
Bieberbach groups as follows. An abstract crystallographic group of rank $n$
is any group which is isomorphic to a uniform discrete subgroup of the
Euclidean group $E(n)=\mathbb{R}^{n}\rtimes O(n)$ of motions on the
Euclidean space $\mathbb{R}^{n}$. An abstract Bieberbach group of dimension $%
n$ is any torsion-free crystallographic group of rank $n.$ The classical
Bieberbach theorems imply that $E$ is an abstract crystallographic group of
rank $n$ if and only if it contains a finite-index normal free abelian group
of rank $n$ which is maximal abelian. The group $E$ is an abstract
Bieberbach group of dimension $n$ if and only if it is a torsion-free
crystallographic group of rank $n$. In both cases, the finite quotient group
acts faithfully on $\mathbb{Z}^{n}$. The quotient group is called the
holonomy group when $E$ is torsion-free. For example, there is a short exact
sequence%
\begin{equation*}
1\rightarrow \mathbb{Z}^{n}\rightarrow \pi _{1}(M)\rightarrow \Phi
\rightarrow 1,
\end{equation*}%
where $\mathbb{Z}^{n}$ is the maximal normal abelian subgroup of $\pi
_{1}(M)<\mathbb{R}^{n}\rtimes O(n)$ and $\Phi $ is the holonomy group. For
an element $g=(a,b)\in \mathbb{R}^{n}\rtimes O(n),$ we call $b$ the rotation
part and $a$ the translation part. For more details, see the book of Charlap 
\cite{ch}.

Let $G^{\ast }$ be the group consisting of all liftings of elements in $G$
to the universal cover $\tilde{M}=\mathbb{R}^{n}$. There is an short exact
sequence%
\begin{equation*}
1\rightarrow \pi _{1}(M)\rightarrow G^{\ast }\rightarrow G\rightarrow 1.
\end{equation*}%
Lee and Raymond \cite{lr} (Prop. 2) prove that $G^{\ast }$ is an abstract
crystallographic group of rank $n$. Moreover, the centralizer $C_{G^{\ast }}(%
\mathbb{Z}^{n})$ is the unique maximal abelian normal subgroup of $G^{\ast }$%
, where $\mathbb{Z}^{n}$ is the maximal normal abelian subgroup of $\pi
_{1}(M).$

\begin{lemma}
\label{flatlem}Let $G$ be a finite group acting effectively on a closed flat
manifold $M.$ There is a commutative diagram%
\begin{equation*}
\begin{array}{ccccccccc}
1 & \rightarrow & \mathbb{Z}^{n} & \rightarrow & C_{G^{\ast }}(\mathbb{Z}%
^{n}) & \rightarrow & A & \rightarrow & 1 \\ 
&  & \downarrow &  & \downarrow &  & f\downarrow &  &  \\ 
1 & \rightarrow & \pi _{1}(M) & \overset{i}{\rightarrow } & G^{\ast } & 
\rightarrow & G & \rightarrow & 1 \\ 
&  & \downarrow &  & \downarrow &  & q\downarrow &  &  \\ 
1 & \rightarrow & \Phi & \overset{g}{\rightarrow } & \Phi ^{\ast } & 
\rightarrow & \Phi ^{\ast }/\Phi & \rightarrow & 1,%
\end{array}%
\end{equation*}%
where $G^{\ast }$ is the group of all liftings of elements in $G$ to the
universal cover $\tilde{M}$. All the rows and columns are exact.
\end{lemma}

\begin{proof}
We already have the middle horizontal exact sequence. Let $\mathbb{Z}^{n}$
be the pure translation subgroup of the Bieberbach group $\pi _{1}(M)$ with
holonomy group $\Phi .$ This gives the first vertical exact sequence. It is
already known that $G^{\ast }$ is isomorphic to an abstract Crystallographic
group with the pure translation subgroup $C_{G^{\ast }}(\mathbb{Z}^{n})$ by 
\cite{lr} (Prop. 2 and its proof). Take $\Phi ^{\ast }$ as the holonomy
group. We then obtain the second vertical exact sequence. Define $A$ as the
quotient group $C_{G^{\ast }}(\mathbb{Z}^{n})/\mathbb{Z}^{n}$ to give the
first horizontal exact sequence. Since $%
\begin{array}{ccc}
\mathbb{Z}^{n} & \rightarrow  & C_{G^{\ast }}(\mathbb{Z}^{n})%
\end{array}%
$ is injective, we have a group homomorphism 
\begin{equation*}
g:\Phi \rightarrow \Phi ^{\ast }.
\end{equation*}
We prove that $g$ is injective as follows. For any $x\in \Phi $ with $%
g(x)=1\in \Phi ^{\ast },$ choose $y\in \pi _{1}(M)$ as a preimage of $x.$
Since $i(y)$ is mapped to the identity in $\Phi ^{\ast },$ we see that $%
i(y)\in C_{G^{\ast }}(\mathbb{Z}^{n}).$ However, the holonomy group $\Phi $
acts effectively on $\mathbb{Z}^{n},$ which implies the rotation part of $%
i(y)$ is trivial. Thus $x=1.$ It is obvious that $g(\Phi )$ is normal in $%
\Phi ^{\ast },$ which gives the third horizontal exact sequence. We then
have a surjective group homomorphism 
\begin{equation*}
q:G\rightarrow \Phi ^{\ast }/\Phi .
\end{equation*}
Since the map $\mathbb{Z}^{n}\rightarrow \pi _{1}(M)$ is injective, there is
an induced group homomorphism $f:A\rightarrow G.$ We prove that $f$ is
injective as follows. For any $x\in A$ with $f(x)=1\in G,$ choose $y\in
C_{G^{\ast }}(\mathbb{Z}^{n})$ as a preimage. Since $y$ is mapped to the
identity in $G,$ there is an element $z\in \pi _{1}(M)$ such that $i(z)=y.$
Since $z$ commutes with $\mathbb{Z}^{n},$ the rotation part of $z$ is
trivial. This implies that $y\in \mathbb{Z}^{n}$ and $x=1.$ The third
vertical sequence is exact at $G\ $by an argument of Snake lemma as follows.
For any $x\in G$ with $q(x)=1\in \Phi ^{\ast }/\Phi ,$ choose $y\in G^{\ast }
$ as a preimage of $x.$ Then the image of $y$ in $\Phi ^{\ast }$ actually
lies in $g(\Phi ).$ Denote the image by $g(z)$ for some $z\in \Phi .$ Choose 
$y_{1}\in \pi _{1}(M)$ as a preimage of $z.$ Then $i(y_{1})$ and $y$ have
the same image in $\Phi ^{\ast }$ and thus $i(y_{1})^{-1}y\in C_{G^{\ast }}(%
\mathbb{Z}^{n}).$ The image of $i(y_{1})^{-1}y$ in $A$ is mapped to $x.$ The
proof is finished.
\end{proof}

\bigskip

Recall from \cite{lr} that an abstract kernel $(G,\pi ,\varphi )$ is a group
homomorphism $\varphi :G\rightarrow \mathrm{Out}(\pi )$ for some fundamental
group $\pi =\pi _{1}(M)$ of a closed flat manifold $M$. A geometric
realization of this abstract kernel is a group homomorphism $\varphi
^{\prime }:G\rightarrow \mathrm{Homeo}(M)$, where $\mathrm{Homeo}(M)$ is the
group of homeomorphisms of $M$, so that $\varphi ^{\prime }$ composed with
the natural homomorphism $\mathrm{Homeo}(M)\rightarrow \mathrm{Out}(\pi )$
agrees with $\varphi $. An extension $E$ of $\pi $ by a group $G$ is said to
be admissible if in the induced diagram%
\begin{equation*}
\begin{array}{ccccccccc}
1 & \rightarrow & \mathbb{\pi } & \rightarrow & E & \rightarrow & G & 
\rightarrow & 1 \\ 
&  & \downarrow &  & \bar{\varphi}\downarrow &  & \varphi \downarrow &  & 
\\ 
1 & \rightarrow & \mathrm{Inn}(\pi ) & \overset{i}{\rightarrow } & \mathrm{%
Aut}(\pi ) & \rightarrow & \mathrm{Out}(\pi ) & \rightarrow & 1,%
\end{array}%
\end{equation*}%
the map $\bar{\varphi}$ is injective on any finite subgroup of $E.$\ Lee and
Raymond \cite{lr} (Theorem 3) prove the following.

\begin{lemma}
\label{lr}Let $M(\pi )$ be a closed Riemannian flat manifold. If an abstract
kernel $(G,\pi ,\varphi )$ admits an admissible extension $E$, then there is
a geometric realization of this extension by an effective affine action of $%
G $ on $M(\pi )$ which is affinely equivalent to an isometric action on an
affinely equivalent flat manifold $M(\theta (\pi ))$. Furthermore, the
lifting of this affine action to $\tilde{M}(\pi )$ induce the same
automorphisms of $\pi $ as $E$.
\end{lemma}

\begin{proof}[Proof of Theorem \protect\ref{th3}]
Suppose that the finite group $G$ acts effectively on the flat manifold $M.$
By Lemma \ref{flatlem}, there is a short exact sequence%
\begin{equation*}
1\rightarrow A\rightarrow G\rightarrow \Phi ^{\ast }/\Phi \rightarrow 1,
\end{equation*}%
where $\Phi $ is the holonomy group of $M$ and $\Phi ^{\ast }$ is the
holonomy group of the lifting group $G^{\ast }.$ Choose $Q=\Phi ^{\ast
}/\Phi .$ Since $G^{\ast }$ is an abstract crystallographic group, the
holonomy group $\Phi ^{\ast }$ acts effectively on $C_{G^{\ast }}(\mathbb{Z}%
^{n}).$ Note that $C_{G^{\ast }}(\mathbb{Z}^{n})$ is isomorphic to $\mathbb{Z%
}^{n}.$ This implies that $\Phi ^{\ast }$ is a subgroup of $\mathrm{GL}_{n}(%
\mathbb{Z}).$ This proves (i). Lemma \ref{flatlem} also implies (ii).

Conversely, suppose that an exact sequence satisfies (i) and (ii). Denote by 
$\pi :=\ker f$ and $N=:\ker \alpha .$ Since $A$ is finite, the group
\thinspace $N$ is isomorphic to $\mathbb{Z}^{n}.$ We prove that $N$ is
normal in $\pi $ as follows. For any $g\in \pi ,n\in N,$ the element $%
gng^{-1}\in G^{\ast }$ is mapped to $1\in \Phi ^{\ast }.$ Therefore, $%
gng^{-1}\in \mathbb{Z}^{n}.$ Since $gng^{-1}$ has image \ $1\in A,$ we get
that $gng^{-1}\in N.$ Denote by $\Phi _{1}:=\pi /N.$ Since $N$ is a subgroup
of $\mathbb{Z}^{n},$ the induced map $\Phi _{1}\rightarrow \Phi ^{\ast }$ is
injective as follows. For any $g\in \Phi _{1}$ with trivial image in $\Phi
^{\ast },$ let $x\in \pi $ be a preimage. Then $x\in \mathbb{Z}^{n}.$ Since $%
x$ is mapped to be identity in $G,$ we see that $x\in N$ and $g=1.$ Now we
have a commutative diagram%
\begin{equation*}
\begin{array}{ccccccccc}
1 & \rightarrow & N & \rightarrow & \mathbb{Z}^{n} & \rightarrow & A & 
\rightarrow & 1 \\ 
&  & \downarrow &  & \downarrow &  & \downarrow &  &  \\ 
1 & \rightarrow & \pi & \rightarrow & G^{\ast } & \rightarrow & G & 
\rightarrow & 1 \\ 
&  & \downarrow &  & \downarrow &  & \downarrow &  &  \\ 
1 & \rightarrow & \Phi _{1} & \rightarrow & \Phi ^{\ast } & \rightarrow & Q
& \rightarrow & 1.%
\end{array}%
\end{equation*}%
We prove the third horizontal sequence is exact at $\Phi ^{\ast }$ as
follows. For any $g\in \Phi ^{\ast }$ with trivial image in $Q,$ let $x\in
G^{\ast }$ be a preimage, which is mapped to be an element $y\in A<G.$
Choose $z\in \mathbb{Z}^{n}$ as a preimage of $y.$ Note that $z^{-1}x\in \pi
.$ The image of $z^{-1}x$ in $\Phi _{1}$ is mapped to be $g.$ This also
proves $\Phi _{1}=\Phi .$

We prove that $\pi $ is the fundamental group of a closed flat manifold \ $%
M^{n}$ with holonomy group $\Phi .$ By a classical result of Auslander and
Kuranishi (cf. \cite{ch}, Theorem 1.1 in Chapter III), it suffices to prove
that $\pi $ is an abstract Bieberbach group with $N$ as its maximal abelian
normal subgroup. This is equivalent to prove that $\Phi $ acts effectively
on $N.$ Suppose that some $1\neq g\in \Phi $ acts trivially on $N.$ Since $A$
is finite, there is a positive integer $k$ such that $ka\in N$ for any $a\in 
\mathbb{Z}^{n}.$ Then $g(ka)=kg(a)$ and thus $g(a)=a,$ which implies that $g$
acts trivially on $\mathbb{Z}^{n}.$ This is a contradiction to the fact that 
$\Phi ^{\ast }$ is a subgroup of $\mathrm{GL}_{n}(\mathbb{Z}).$

We check that the middle horizontal exact sequence is an admissible
extension of $\pi $ as follows. Since $\pi $ is normal in $G^{\ast },$ there
is a commutative diagram%
\begin{equation*}
\begin{array}{ccccccccc}
1 & \rightarrow & \pi & \overset{i}{\rightarrow } & G^{\ast } & \rightarrow
& G & \rightarrow & 1 \\ 
&  & \downarrow &  & \downarrow \phi &  & \downarrow &  &  \\ 
1 & \rightarrow & \mathrm{Inn}(\pi ) & \rightarrow & \mathrm{Aut}(\pi ) & 
\rightarrow & \mathrm{Out}(\pi ) & \rightarrow & 1.%
\end{array}%
\end{equation*}%
Note that $\ker \phi =C_{G^{\ast }}(\pi ),$ the centralizer of $\pi .$ We
claim that $\ker \phi $ is a subgroup of $\mathbb{Z}^{n}<G^{\ast }.$ Suppose
that there is an element $g\in C_{G^{\ast }}(\pi )\backslash \mathbb{Z}^{n}.$
The image $\bar{g}$ of $g$ in $\Phi ^{\ast }$ is not trivial. Since $g$
commutes with elements in $\pi ,$ the action of $g$ on $N$ is trivial. The
same argument as that in the previous paragraph shows that the action of $g$
on $\mathbb{Z}^{n}$ is trivial. This is a contradiction to the fact that $%
\Phi ^{\ast }$ acts effectively on $\mathbb{Z}^{n}.$ Since $\ker \phi $ is
torsion-free, the exact sequence is an admissible extension. By Lemma \ref%
{lr}, the group $G$ acts effectively on the closed flat manifold $M$ with
holonomy group $\Phi .$ The proof is finished.
\end{proof}

For finite groups acting on tori, we have a simple characterization.

\begin{theorem}
\label{torus}A finite group $G$ acts effectively on a torus $T^{n}$ if and
only if there is an abelian extension 
\begin{equation*}
1\rightarrow A\rightarrow G\rightarrow Q\rightarrow 1
\end{equation*}%
such that

\begin{enumerate}
\item[(i)] $Q<\mathrm{GL}_{n}(\mathbb{Z});$

\item[(ii)] there is an $Q$-equivariant surjection $\ \alpha :\mathbb{Z}%
^{n}\twoheadrightarrow A$ and the cohomology class representing of the
extension lies in the image $\func{Im}(H^{2}(Q;\mathbb{Z}^{n})\rightarrow
H^{2}(Q;A)).$
\end{enumerate}
\end{theorem}

\begin{proof}
\bigskip The necessary condition follows Theorem \ref{th3} easily.
Conversely, suppose that there is a commutative diagram%
\begin{equation*}
\begin{array}{ccccccccc}
1 & \rightarrow & \mathbb{Z}^{n} & \rightarrow & G^{\ast } & \rightarrow & Q
& \rightarrow & 1 \\ 
&  & \alpha \downarrow &  & f\downarrow &  & =\downarrow &  &  \\ 
1 & \rightarrow & A & \overset{i}{\rightarrow } & G & \rightarrow & Q & 
\rightarrow & 1.%
\end{array}%
\end{equation*}%
Note that $\ker f\cong \ker \alpha \cong \mathbb{Z}^{n}$ is a torsion-free
group. The proof is finished by applying Theorem \ref{th3} again.
\end{proof}

\begin{corollary}
\label{a4}\bigskip Let $A_{4}$ be the alternating group. Then $d_{\mathbb{Z}%
}(A_{4})=3$ and $d_{h}(A_{4},\mathcal{FM})=2.$
\end{corollary}

\begin{proof}
Since the minimal nontrivial degree of irreducible representations of $A_{4}$
is $3$ and $A_{4}$ is a subgroup of $\mathrm{SL}_{3}(\mathbb{Z}),$ we get $%
d_{\mathbb{Z}}(A_{4})=3.$ Since $A_{4}$ is not isomorphic to a subgroup of a
dihedral group, we have $d_{h}(A_{4},\mathcal{FM})\geq 2,$ since any finite
group acting effectively on $S^{1}$ is either cyclic or dihedral. Note that $%
A_{4}\cong (\mathbb{Z}/2)^{2}\rtimes \mathbb{Z}/3,$ where $(\mathbb{Z}%
/2)^{2}=\langle (14)(23),(13)(24)\rangle $ and $\mathbb{Z}/3=\langle
(123)\rangle ,$ where $\mathbb{Z}/3$ acts on $(\mathbb{Z}/2)^{2}$ through
matrix 
\begin{equation*}
\begin{pmatrix}
0 & 1 \\ 
1 & 1%
\end{pmatrix}%
\in \mathrm{GL}_{2}(\mathbb{Z}/2).
\end{equation*}%
Take $Q=\mathbb{Z}/3.$ Define $G^{\ast }=\mathbb{Z}^{2}\rtimes \mathbb{Z}/3,$
where $\mathbb{Z}/3$ acts on the free abelian group $\mathbb{Z}^{2}$ through
matrix 
\begin{equation*}
\begin{pmatrix}
0 & -1 \\ 
1 & -1%
\end{pmatrix}%
\in \mathrm{GL}_{2}(\mathbb{Z}).
\end{equation*}%
Let 
\begin{equation*}
\alpha :\mathbb{Z}^{2}\rightarrow \mathbb{Z}/2
\end{equation*}
be the modulo 2 map. It is not hard to see that $\alpha $ is $Q$%
-equivariant. Moreover, the map $\alpha $ induces a map between the two
split extensions. Theorem \ref{torus} implies that $A_{4}$ acts effectively
on $T^{2},$ which proves $d_{h}(A_{4},\mathcal{FM})=2.$
\end{proof}

\bigskip 

\begin{proof}[Proof of Corollary \protect\ref{th4}]
By Theorem \ref{th3}, the group $G$ contains a normal abelian subgroup $A$
such that the quotient $G/A$ is a quotient group of a finite group in $%
\mathrm{GL}_{n}(\mathbb{Z}).$ Therefore, the cardinality $|G/A|\leq f(n).$
\end{proof}

\section{Actions of simple groups on flat manifolds}

\begin{lemma}
\label{simple}Let $G$ be a noncommutative simple group and $M^{n}$ a closed
flat manifold. Suppose that $G$ acts on $M$ effectively by homeomorphisms.
Then $G$ is isomorphic to the quotient group of a finite subgroup in $%
\mathrm{GL}_{n}(\mathbb{Z})$ by the holonomy group of $M.$
\end{lemma}

\begin{proof}
By Theorem \ref{th3}, there is an exact sequence 
\begin{equation*}
1\rightarrow A\rightarrow G\rightarrow \Phi ^{\ast }/\Phi \rightarrow 1.
\end{equation*}%
Since $G$ is a noncommutative simple group, the normal abelian subgroup $A$
is trivial. Therefore, $G$ is isomorphic to $\Phi ^{\ast }/\Phi .$
\end{proof}

\begin{corollary}
If a noncommutative simple group $G$ contains an element of prime order $%
p>n+1,$ then any group action of $G$ on a closed flat $n$-dimensional
manifold $M$ is trivial.
\end{corollary}

\begin{proof}
It is well known that the prime order of elements in $\mathrm{GL}_{n}(%
\mathbb{Z})$ is at most $n+1$ (cf. \cite{New}, p. 181, exercise 1). Lemma %
\ref{simple} proves the statement.
\end{proof}

Our main result in this section is the following.

\begin{theorem}
\label{th1}Let $G=A_{n+1}$ $(n\geq 4)$ be the alternating group. Then%
\begin{equation*}
d_{\mathbb{Z}}(G)=d_{h}(G,\mathcal{FM})=n,
\end{equation*}%
i.e. the minimal acting dimension $d_{h}(G,\mathcal{FM})$ of $G$ on closed
flat manifolds is the same as the minimal faithful integral representation
dimension $d_{\mathbb{Z}}(G).$
\end{theorem}

Since \ $A_{5}$ is a subgroup of $\mathrm{SL}_{3}(\mathbb{R})$ which acts
effectively on the Euclidean space $\mathbb{R}^{3}$ linearly, it is
impossible to extend Theorem \ref{th1} to either $d_{\mathbb{R}}$ or $d_{h}$
on \emph{all} (compact and non-compact) flat manifolds. The bound on $n$
cannot be improved, since $d_{\mathbb{Z}}(A_{4})=3$ and $d_{h}(A_{4},%
\mathcal{FM})=2$ (cf. Corollary \ref{a4}).

\bigskip

In order to prove Theorem \ref{th1}, we need the following result of M.
Collins \cite{co} (Theorem B).

\begin{lemma}
\label{col}Let $G$ be a finite subgroup of the general linear group $\mathrm{%
GL}_{n}(\mathbb{R}).$ If $n\geq 25,$ then $G$ has an abelian normal subgroup 
$A$ of index at most $(n+1)!$.
\end{lemma}

\begin{proof}[Proof of Theorem \protect\ref{th1}]
By Lemma \ref{le1}, we have $d_{h}(G,\mathcal{FM})\leq d_{\mathbb{Z}}(G).$
Since $A_{n+1}$ is a subgroup of $\mathrm{GL}_{n}(\mathbb{Z}),$ we get that $%
d_{\mathbb{Z}}(G)\leq n.$ It suffices to prove that $n\leq d_{h}(G,\mathcal{%
FM}).$ Let $G$ act effectively on a closed flat manifold $M^{k}$ with $%
k=d_{h}(G,\mathcal{FM}).$ When $n\geq 4,$ the group $G=A_{n+1}$ is simple.
Lemma \ref{simple} implies that $G=K/H$ for some finite subgroup $K<\mathrm{%
GL}_{k}(\mathbb{Z})$ and a normal subgroup $H\trianglelefteq K.$ Since $G$
is noncommutative simple, any abelian normal subgroup of $K$ is contained in 
$H.$ We prove the theorem in two cases.

\begin{enumerate}
\item[Case (i)] $n\geq 25.$ 

By Lemma \ref{col}, the size of $G$ is at most $(k+1)!,$ which implies that $%
n=k.$

\item[Case (ii)] $4\leq n<25.$ 

Let $K^{\prime }$ be a maximal finite subgroup of $\mathrm{GL}_{n}(\mathbb{Q}%
)$ containing $K.$ We have that the cardinality $|K^{\prime }|$ is divisible
by $|G|=(n+1)!/2.$ Note that each maximal finite subgroup $K^{\prime }<%
\mathrm{GL}_{n}(\mathbb{Q})$ is isomorphic to a product $G_{1}\times
G_{2}\times \cdots \times G_{s},$ where $G_{i}$ is irreducible maximal
finite subgroups of $\mathrm{GL}_{n_{i}}(\mathbb{Q})$ for $i=1,...,s$ and $%
n_{1}+n_{2}+\cdots +n_{s}=n$ (cf. \cite{Pl}, (11.4) Remark (i), page 479).
When $k<31,$ all the irreducible maximal finite subgroup of $\mathrm{GL}_{k}(%
\mathbb{Q})$ are classified in \cite{PN95, np95, ne95,ne}. Suppose that $%
k\leq n-1.$ We check these maximal subgroups to get a contradiction.
Practically, we use the software GAP \cite{gap} to do this. First, input the
following code:
\end{enumerate}

\texttt{gap\TEXTsymbol{>} d:=[];;s:=[];;k:=1;;}

\texttt{gap\TEXTsymbol{>} for n in [3..24] do}

\texttt{\TEXTsymbol{>} A:=Partitions(n);}

\texttt{\TEXTsymbol{>} for j in [1..NrPartitions(n)] do}

\texttt{\TEXTsymbol{>} B:=A[j];}

\texttt{\TEXTsymbol{>} orders:=List([1..Size(B)], x-\TEXTsymbol{>}%
List([1..ImfNumberQQClasses(B[x])],}

\texttt{\ y-\TEXTsymbol{>}ImfInvariants(B[x], y).size ));}

\texttt{\TEXTsymbol{>} C:=Cartesian(orders);}

\texttt{\TEXTsymbol{>} for i in [1..Size(C)] do}

\texttt{\TEXTsymbol{>} prod:=Product(C[i]) mod Size(AlternatingGroup(n+2));}

\texttt{\TEXTsymbol{>} if prod=0 then}

\texttt{\TEXTsymbol{>} d[k]:=B;s[k]=C[i];k:=k+1;}

\texttt{\TEXTsymbol{>} fi;}

\texttt{\TEXTsymbol{>} od;}

\texttt{\TEXTsymbol{>} od;}

\texttt{\TEXTsymbol{>} od;}

\texttt{gap\TEXTsymbol{>} d;s;}

\texttt{[ [ 7 ], [ 8 ] ]}

\texttt{[ [ 2903040 ], [ 696729600 ] ]}\textrm{.}

\texttt{gap\TEXTsymbol{>} List([1..ImfNumberQQClasses(7)],x-\TEXTsymbol{>}%
ImfInvariants(7, x).size) mod 2903040;}

\texttt{\ List([1..ImfNumberQQClasses(8)],x-\TEXTsymbol{>}ImfInvariants(8,
x).size) mod 696729600;}

\texttt{[ 645120, 0 ]}

\texttt{[ 10321920, 2654208, 0, 6912, 497664, 115200, 28800, 1440, 672 ].}

The program finds the maximal groups in $\mathrm{GL}_{n}(\mathbb{Q})$ whose
orders are divisible by $|G|=(n+2)!/2$ for each \ $n\leq 24.$ The output
shows that there are only two maximal finite groups $K_{1}^{\prime }=$%
\textrm{ImfMatrixGroup}$(7,2,1)$ and $K_{2}^{\prime }=$\textrm{ImfMatrixGroup%
}$(8,3,1)$ whose orders could be divisible by those of $A_{9}$ and $A_{10}$,
respectively. Here \textrm{ImfMatrixGroup}$(n,i,1)$ represents the $\ i$-th
irreducible maximal finite group in $\mathrm{GL}_{n}(\mathbb{Q}).$ However, $%
K_{1}^{\prime }$ is isomorphic to the Weyl group of $E_{7},$ and $%
K_{2}^{\prime }\ $is isomorphic to the Weyl group of $E_{8}$ (this could be
seen from the commands \texttt{DisplayImfInvariants(7, 2, 1)} and \texttt{%
DisplayImfInvariants(8, 3, 1)} in GAP). Input the following code:

\texttt{gap\TEXTsymbol{>} s:=[];;j:=1;;}

gap\TEXTsymbol{>} \texttt{%
cc:=ConjugacyClassesSubgroups(ImfMatrixGroup(7,2,1));;Size(cc);}

\texttt{gap\TEXTsymbol{>}for i in [1..Size(cc)] do}

\texttt{\TEXTsymbol{>} a:=Size(Representative(cc[i])) mod
Size(AlternatingGroup(9));}

\texttt{\TEXTsymbol{>} if a=0 then}

\texttt{\TEXTsymbol{>} s[j]:=i;j:=j+1;}

\texttt{\TEXTsymbol{>} fi;}

\texttt{\TEXTsymbol{>} od;}

\texttt{gap\TEXTsymbol{>} s;}

\texttt{[ 8073, 8074 ].}

\texttt{gap\TEXTsymbol{>} GQuotients(Representative(cc[8073]),
AlternatingGroup(9));}

\TEXTsymbol{>} \texttt{GQuotients(Representative(cc[8074]),
AlternatingGroup(9));}

\texttt{[ ] }

\texttt{[ ].}

The program finds the subgroups of $K_{1}^{\prime }$ which have nontrivial
surjections to $A_{9}.$ The output shows that among the 8074 conjugacy
classes of subgroups in $K_{1}^{\prime },$ there are only two classes whose
orders could be divisible by that of $A_{9}$. Furthermore, neither of these
two classes of groups could have a quotient group $A_{9}.$

A similar argument proves that $K_{2}^{\prime }$ does not contain a subgroup
whose quotient is isomorphic to $A_{10}.$ Note that we can not apply
directly the same code used for dealing with \textrm{ImfMatrixGroup}$(7,2,1)$%
, since the group \textrm{ImfMatrixGroup}$(8,3,1)$ is too large to compute
(in an ordinary laptop). We proceed as follows. First note that the center $%
Z $ of the Weyl group $K_{2}^{\prime }$ of $E_{8}$ is of order two and the
quotient group $K_{2}^{\prime }/Z$ is isomorphic to the orthogonal group $%
O_{8}^{+}(2),$ the linear transformations of an 8-dimensional vector space
over the two-element field $\mathbb{Z}/2$ preserving a quadratic form of
plus type (cf. \cite{ct}, p.85). Since the alternating group $A_{10}$ is
simple, the subgroup $K<K_{2}^{\prime }$ with $K/H\cong A_{10}$ could chosen
to be a subgroup in $O_{8}^{+}(2).$ The group $O_{8}^{+}(2)$ has no quotient
group isomorphic to $A_{10},$ by checking the following code in GAP:

\texttt{gap\TEXTsymbol{>} o:=GO(1,8,2);;}

\texttt{gap\TEXTsymbol{>} GQuotients(o,AlternatingGroup(10));}

\texttt{[ ].}

Therefore, the subgroup $K$ lies in the (indexed 2) unique proper maximal
normal subgroup $N$ of $O_{8}^{+}(2).$ By a similar way, we see that $N$ has
no quotient groups isomorphic to $A_{10}.$ This implies that $K$ lies in a
maximal subgroup of $N.$ Input the following code in GAP:

\texttt{gap\TEXTsymbol{>} o:=GO(1,8,2);;}

\texttt{gap\TEXTsymbol{>} n:=MaximalNormalSubgroups(o);;}

\texttt{gap\TEXTsymbol{>}
cc:=ConjugacyClassesMaximalSubgroups(n[1]);;Size(cc);}

\texttt{gap\TEXTsymbol{>} s:=[];;j:=1;;}

\texttt{gap\TEXTsymbol{>} for i in [1..Size(cc)] do}

\texttt{\TEXTsymbol{>} a:=Size(Representative(cc[i])) mod
Size(AlternatingGroup(10));}

\texttt{\TEXTsymbol{>} if a=0 then}

\texttt{\TEXTsymbol{>} s[j]:=i;j:=j+1;}

\texttt{\TEXTsymbol{>} fi;}

\texttt{\TEXTsymbol{>} od;}

\texttt{gap\TEXTsymbol{>} s;}

\texttt{[ ].}

The program finds all the subgroups of $N$ whose orders are divisible by the
order $|A_{10}|.$ The output shows that there are no such subgroups. The
whole proof is finished.
\end{proof}

\section{Groups normally generated by alternating groups\label{sec}}

In this section, we consider typical (infinite) groups normally generated by
alternating groups.

\subsection{Automorphism groups of free groups}

Let $F_{n}=\langle a_{1},\cdots ,a_{n}\rangle $ be a free group of $n$
letters. Denote by $\mathrm{Aut}(F_{n})$ the automorphism group of $F_{n}$
and by $\mathrm{Out}(F_{n})=\mathrm{Aut}(F_{n})/\mathrm{Inn}(F_{n})$ the
outer automorphism group, where $\mathrm{Inn}(F_{n})$ is the inner
automorphism subgroup. For $\ 1\leq i\neq j\leq n,$ define $\sigma _{ij}\in 
\mathrm{Aut}(F_{n})$ as 
\begin{equation*}
\sigma _{ij}(a_{i})=a_{j},\sigma _{ij}(a_{j})=a_{i}\text{ and }\sigma
_{ij}(a_{k})=a_{k}\text{, }k\neq i,j.
\end{equation*}%
Furthermore, define $\sigma _{i,n+1}\in \mathrm{Aut}(F_{n})$ as 
\begin{equation*}
\sigma _{i,n+1}(a_{i})=a_{i}^{-1}\text{ and }\sigma
_{i,n+1}(a_{j})=a_{j}a_{i}^{-1},j\neq i.
\end{equation*}%
The subgroup of $\mathrm{Aut}(F_{n})$ generated by the elements $\sigma
_{ij} $ and $\sigma _{i,n+1}$ is isomorphic to the symmetric group $S_{n+1}$
of $n+1$ letters, in such a way that $\sigma _{ij}$ corresponds to the
permutation $(ij)$ (cf. \cite{bv2}, Section 6). The action of $\mathrm{Aut}%
(F_{n})$ on the abelianization of $F_{n}$ induces a homomorphism from $%
\mathrm{Aut}(F_{n})$ to $\mathrm{GL}_{n}(\mathbb{Z})$ that factors through
the outer automorphism group $\mathrm{Out}(F_{n})$. The inverse images of
the special linear group $\mathrm{SL}_{n}(\mathbb{Z})$ under these maps are
normal subgroups denoted here by the special automorphism group $\mathrm{SAut%
}(F_{n})$ and $\mathrm{SOut}(F_{n})$ respectively.

\begin{lemma}
\label{autf}When \ $n\geq 3,$ the groups $\mathrm{SAut}(F_{n})$ and $\mathrm{%
SOut}(F_{n})$ are normally generated by the alternating group $A_{n+1}.$
\end{lemma}

\begin{proof}
It is actually proved by Berrick and Matthey \cite{Ber} (Lemma 2.2) that the
group $\mathrm{SAut}(F_{n})$ is normally generated by \ $A_{n},$ which
clearly implies the statement.
\end{proof}

\subsection{General linear groups}

Let $R$ be an associative ring (may be not abelian) with identity. The
general linear group $\mathrm{GL}_{n}(R)$ is the group of all $n\times n$
invertible matrices with entries in $R$. For an element $r\in R$ and any
integers $i,j$ such that $1\leq i\neq j\leq n,$ denote by $e_{ij}(r)$ the
elementary $n\times n$ matrix with $1s$ in the diagonal positions and $r$ in
the $(i,j)$-th position and zeros elsewhere. The group $E_{n}(R)$ is
generated by all such $e_{ij}(r),$\textsl{\ i.e. }%
\begin{equation*}
E_{n}(R)=\langle e_{ij}(r)|1\leq i\neq j\leq n,r\in R\rangle .
\end{equation*}

\begin{lemma}
\label{lin}When $n\geq 3,$ the group \ $E_{n}(R)$ is normally generated by
the alternating group $A_{n+1}.$
\end{lemma}

\begin{proof}
When \ $R=\mathbb{Z},$ the ring of integers, we get\ $E_{n}(R)=\mathrm{SL}%
_{n}(\mathbb{Z}).$ By Lemma \ref{autf}, $\mathrm{SAut}(F_{n})$ is normally
generated by $\ A_{n+1}.$ Since $A_{n+1}$ is mapped injectively to a
subgroup of $\mathrm{SL}_{n}(\mathbb{Z}),$ we see that $\mathrm{SL}_{n}(%
\mathbb{Z})$ is generated by $\ A_{n+1}.$ For a general ring $R,$ let \ 
\begin{equation*}
i:\mathbb{Z}\rightarrow R
\end{equation*}%
be the natural map defined by $i(1)=1\in R.$ Since $E_{n}(R)$ is normally
generated by the image $i(\mathrm{SL}_{n}(\mathbb{Z}))$ by the commutator
formula (for example, see \cite{M}, 1.2C), the group $E_{n}(R)$ is normally
generated by $A_{n+1}.$
\end{proof}

\subsection{Classical groups}

Let $R$ be an arbitrary ring and assume that an anti-automorphism $\ast
:x\mapsto x^{\ast }$ is defined on $R$ such that $x^{\ast \ast }=\varepsilon
x\varepsilon ^{\ast }$ for some unit $\varepsilon ^{\ast }=\varepsilon ^{-1}$
of $R$ and every $x$ in $R$. It determines an anti-automorphism of the ring $%
M_{n}(R)$ of all $n\times $ $n$ matrices $(x_{ij})$ by $(x_{ij})^{\ast
}=(x_{ji}^{\ast })$. \ Set 
\begin{equation*}
R_{\varepsilon }=\{x-x^{\ast }\varepsilon |\,\,x\in R\}
\end{equation*}
and 
\begin{equation*}
R^{\varepsilon }=\{x\in R\,|\,\,x=-x^{\ast }\varepsilon \}.
\end{equation*}
If some additive subgroup $\Lambda $ of $(R,+)$ satisfies:{}{}

\begin{enumerate}
\item[(i)] $r^{\ast }\Lambda r\subset \Lambda $ for all $r\in R$;

\item[(ii)] $R_{\varepsilon }\subset \Lambda \subset R^{\varepsilon },$
\end{enumerate}

\noindent we will call $\Lambda $ a form and $(\Lambda ,\ast ,\varepsilon )$
a form parameter on $R$. Usually $(R,\Lambda )$ is called a form ring. Let 
\begin{equation*}
\Lambda _{n}=\{(a_{ij})\in M_{n}R|\,\,a_{ij}=-a_{ji}^{\ast }\varepsilon 
\text{ for }i\neq j\text{ and}a_{ii}\in \Lambda \}.
\end{equation*}
For an integer $n\geq 1,$ we define the unitary group\newline
\begin{equation*}
U_{2n}(R,\Lambda )=\left\{ \left( 
\begin{array}{ll}
\alpha & \beta \\ 
\gamma & \delta%
\end{array}%
\right) \in \mathrm{GL}_{2n}R\,|\,\,\alpha ^{\ast }\delta +\gamma ^{\ast
}\varepsilon \beta =I_{n},\,\,\,\alpha ^{\ast }\gamma ,\,\,\,\beta ^{\ast
}\delta \in \Lambda _{n}\right\} .
\end{equation*}%
We could also define the elementary unitary group $EU_{2n}(R,\Lambda )$. For
more details, see \cite{M}.

The unitary group $U_{2n}(R,\Lambda )$ has many important special cases, as
follows.

\begin{itemize}
\item When $\Lambda =R$, the group $U_{2n}(R,\Lambda )$ is the symplectic
group. This can only happen when $\varepsilon =-1$ and $\ast =\mathrm{id}%
_{R} $ ($R$ is commutative) is the trivial anti-automorphism.

\item When $\Lambda =0$, the group $U_{2n}(R,\Lambda )$ is the ordinary
orthogonal group. This can only happen when $\varepsilon =1$ and $\ast =%
\mathrm{id}_{R}\ $($R$ is commutative) as well.

\item When $\Lambda =R^{\varepsilon }$ and $\ast \neq \mathrm{id}_{R}$, the
group $U_{2n}(R,\Lambda )$ is the classical unitary group 
\begin{equation*}
U_{2n}=\{A\in \mathrm{GL}_{2n}R|\,\ A^{\ast }\varphi _{n}A=\varphi _{n}\},
\end{equation*}%
where 
\begin{equation*}
\varphi _{n}=\left( 
\begin{array}{cc}
0 & I_{n} \\ 
\varepsilon I_{n} & 0%
\end{array}%
\right) .
\end{equation*}
\end{itemize}

\begin{lemma}
\label{ui}When \ $n\geq 3,$ the elementary group $EU_{2n}(R,\Lambda )$ is
normally generated by the alternating group $A_{n+1}.$
\end{lemma}

\begin{proof}
The hyperbolic homomorphism $H:\mathrm{GL}_{n}(R)\rightarrow
U_{2n}(R,\Lambda )$ defined by 
\begin{equation*}
A\longmapsto \left( 
\begin{array}{cc}
A &  \\ 
& (A^{-1})^{\ast }%
\end{array}%
\right)
\end{equation*}%
induced an embedding $E_{n}(R)\rightarrow EU_{2n}(R,\Lambda )$ (cf. \cite{M}%
, Section 5.3C). By the commutator formula for unitary groups, the group $%
EU_{2n}(R,\Lambda )$ is normally generated by $E_{n}(R)$ (cf. \cite{M},
5.3.13 and 5.3B). Since $E_{n}(R)$ is normally generated by $A_{n+1}$ by
Lemma \ref{lin}, the group $EU_{2n}(R,\Lambda )$ is normally generated by $%
A_{n+1}.$
\end{proof}

\section{Proof of Theorem \protect\ref{th2}}

In order to prove Theorem \ref{th2}, we need the following lemma of
Weinberger \cite{we} (Lemma 2).

\begin{lemma}
\label{wein}If a finite group $H$ acts homologically trivially on a torus \ $%
T^{r}$, then the action is equivariantly homotopy equivalent to an action
that factors through the group of translations of the torus. In particular,
if $H$ is nonabelian, the action is not effective.
\end{lemma}

\begin{proof}[Proof of Theorem \protect\ref{th2}]
When $n\geq 4,$ the alternating group $A_{n+1}$ acts trivially on $M^{r}$ by
Theorem \ref{th1} and Lemma \ref{ke}. Since $G$ is normally generated by $%
A_{n+1}$ (cf. Lemma \ref{autf},\ref{lin},\ref{ui}), any action of \ $G$ on \ 
$M^{r}$ by homeomorphisms is trivial. When $n=3,r=1,$ Theorem \ref{th2} is
already contained in the results proved by Bridson and Vogtmann \cite{bv1}
and Ye \cite{Ye12}. When $n=3,r=2,$ it is proved in \cite{ye} that any
action of $\mathrm{SL}_{n}(\mathbb{Z})$ on $M^{r}$ is trivial. Therefore,
any action of $E_{n}(R)$ or $EU_{2n}(R,\Lambda )$ is trivial since these two
groups are normally generated by the image of $\mathrm{SL}_{n}(\mathbb{Z})$
(see the proofs of Lemma \ref{lin} and Lemma \ref{ui}). For the case when $G=%
\mathrm{SAut}(F_{n}),$ let $A_{4}<G$ be the alternating group subgroup
constructed in Section \ref{sec}. Note that $M^{2}$ is the torus $T^{2}$, or
the Klein bottle $K$. Since any action of $G$ on the nonorientable manifold $%
K$ is uniquely lifted to be an action on the orientable double covering \ $%
T^{2}$ (cf. \cite{Br}, Cor. 9.4, p.67), we would have an action of $G$ on $%
T^{2}$ in both cases of $M$. Since $G$ acts trivially on $H_{2}(T^{2};%
\mathbb{Z})=\mathbb{Z}^{2}$ (cf. \cite{bv1}), Lemma \ref{wein} implies that
the element $\sigma _{12}\sigma _{34}\in \lbrack A_{4},A_{4}]$\textrm{\ }%
(the commutator subgroup) acts trivially on $M.$ However, the group $G$ is
normally generated by $\sigma _{12}\sigma _{34}$ (cf. \cite{bv1},
Proposition 3.1), which implies that the action of $G$ is trivial. This
argument works for all $G$ as well. The proof is finished.
\end{proof}

\begin{corollary}
When $n\geq 3,$ we have $n=d_{\mathbb{Z}}(\mathrm{SL}_{n}(\mathbb{Z}))=d_{%
\mathbb{R}}(\mathrm{SL}_{n}(\mathbb{Z}))=d_{h}(\mathrm{SL}_{n}(\mathbb{Z}),%
\mathcal{FM}).$
\end{corollary}

\begin{proof}
Note that $\mathrm{SL}_{n}(\mathbb{Z})$ acts effectively on the Euclidean
space $\mathbb{R}^{n}$ and the torus $T^{n}.$ This implies 
\begin{equation*}
d_{\mathbb{R}}(\mathrm{SL}_{n}(\mathbb{Z}))\leq d_{\mathbb{Z}}(\mathrm{SL}%
_{n}(\mathbb{Z}))\leq n.
\end{equation*}
The fact that $n\leq d_{\mathbb{R}}(\mathrm{SL}_{n}(\mathbb{Z}))$ follows
from \cite{bv1}. Theorem \ref{th3} implies that $d_{h}(\mathrm{SL}_{n}(%
\mathbb{Z}),\mathcal{FM})=n$.
\end{proof}

\bigskip

\noindent \textbf{Acknowledgements}

The author would like to thank Prof. F. Thomas Farrell for an inspiring
discussion on symmetries of flat manifolds. The author also wants to thank
the referee for a detailed report.

\bigskip

Department of Mathematical Sciences, Xi'an Jiaotong-Liverpool University,
111 Ren Ai Road, Suzhou, Jiangsu 215123, China.

E-mail: Shengkui.Ye@xjtlu.edu.cn

\end{document}